\documentclass[a4paper,11pt,reqno]{amsart}
\usepackage{amscd,amssymb,graphicx,color,epigraph}

\def\bP{\Bbb P}

\def\tY{\tilde{Y}}
\def\bY{{\Bbb Y}}

\def\cO{\Cal O}

\def\codim{\operatorname{codim}}

\def\bZ{\Bbb Z}

\def\bQ{\Bbb Q}
\def\bG{\Bbb G}
\def\bR{\Bbb R}
\def\bA{\Bbb A}

\def\oX{\overline{X}}

\def\cF{\Cal F}
\def\tS{\tilde S}

\def\oS{\overline{S}}

\def\tS{\tilde{S}}

\def\oN{\overline{N}}

\def\cY{\Cal Y}

\def\cX{\Cal X}

\def\cM{\Cal M}

\def\Spec{\operatorname{Spec}}

\def\tM{\tilde{M}}

\def\oG{\overline G}

\def\ord{\operatorname{ord}}

\def\st#1{#1_{st}}

\def\cO{\Cal O}

\def\Spec{\operatorname{Spec}}

\def\pr{\operatorname{pr}}
\def\Hom{\operatorname{Hom}}
\def\Cal{\mathcal}

\def\arrow{\mathop{\longrightarrow}\limits}

\def\Efi#1#2#3#4#5{\displaystyle
#1\!\!-\!\!#2
\!\!-\!\!#3
\!\!-\!\!#4
\hskip-24.2pt\lower4.5pt\hbox{${\scriptstyle|}
\hskip-3.35pt\lower6pt\hbox{$#5$}$}}

\def\Evia#1#2#3#4#5{\displaystyle
#1\!\!-\!\!#2
\!\!-\!\!#3
\hskip-24.2pt\lower4.5pt\hbox{${\scriptstyle|}
\hskip-3.35pt\lower6pt\hbox{$#4\!\!-\!\!\!-\!\!\!-\!\!$}$\hskip2.3pt${\scriptstyle|}
\hskip-3.35pt\lower6pt\hbox{$#5$}$}}

\def\Ezia#1#2#3#4{\displaystyle
#1\!\!-\!\!#2
\hskip-14.8pt\lower4.5pt\hbox{${\scriptstyle|}
\hskip-3.35pt\lower6pt\hbox{$#3\!\!-\!\!$}$\hskip2.3pt${\scriptstyle|}
\hskip-3.35pt\lower6pt\hbox{$#4$}$}}

\def\Efia#1#2#3#4#5#6{\displaystyle
#1\!\!-\!\!#2
\!\!-\!\!#3
\!\!-\!\!#4
\hskip-24.2pt\lower4.5pt\hbox{${\scriptstyle|}
\hskip-3.35pt\lower6pt\hbox{$#5$}$\hskip5.7pt${\scriptstyle|}
\hskip-3.35pt\lower6pt\hbox{$#6$}$}}

\def\Esi#1#2#3#4#5#6{\displaystyle
#1\!\!-\!\!#2
\!\!-\!\!#3
\!\!-\!\!#4\!\!-\!\!#5
\hskip-24.2pt\lower4.5pt\hbox{${\scriptstyle|}
\hskip-3.35pt\lower6pt\hbox{$#6$
\lower3pt\hbox{\ }}$}}

\def\Esia#1#2#3#4#5#6#7{\displaystyle
#1\!\!-\!\!#2
\!\!-\!\!#3
\!\!-\!\!#4\!\!-\!\!#5
\hskip-24.2pt\lower4.5pt\hbox{${\scriptstyle|}
\hskip-3.35pt\lower6pt\hbox{$#6$\hskip-3.8pt\lower4.5pt\hbox{${\scriptstyle|}
\hskip-3.35pt\lower6pt\hbox{$#7$}$}}
\lower3pt\hbox{\ }$}}

\def\Ese#1#2#3#4#5#6#7{\displaystyle
#1\!\!-\!\!#2
\!\!-\!\!#3
\!\!-\!\!#4\!\!-\!\!#5\!\!-\!\!#6
\hskip-33.6pt\lower4.5pt\hbox{${\scriptstyle|}
\hskip-3.35pt\lower6pt\hbox{$#7$
\lower3pt\hbox{\ }
}$}}

\def\Esea#1#2#3#4#5#6#7#8{\displaystyle
#1\!\!-\!\!#2
\!\!-\!\!#3
\!\!-\!\!#4\!\!-\!\!#5\!\!-\!\!#6\!\!-\!\!#7
\hskip-33.6pt\lower4.5pt\hbox{${\scriptstyle|}
\hskip-3.35pt\lower6pt\hbox{$#8$
\lower3pt\hbox{\ }
}$}}

\def\Eei#1#2#3#4#5#6#7#8{\displaystyle
#1\!\!-\!\!#2
\!\!-\!\!#3
\!\!-\!\!#4\!\!-\!\!#5\!\!-\!\!#6\!\!-\!\!#7
\hskip-43.2pt\lower4.5pt\hbox{${\scriptstyle|}
\hskip-3.35pt\lower6pt\hbox{$#8$
\lower3pt\hbox{\ }
}$}}

\def\Eeia#1#2#3#4#5#6#7#8#9{{\displaystyle
#1\!\!-\!\!#2
\!\!-\!\!#3
\!\!-\!\!#4\!\!-\!\!#5\!\!-\!\!#6\!\!-\!\!#7\!\!-\!\!#8
\hskip-52.2pt\lower4.5pt\hbox{${\scriptstyle|}
\hskip-3.35pt\lower6pt\hbox{$#9$
\lower3pt\hbox{\ }
}$}}}

\def\arrow{\mathop{\longrightarrow}\limits}

\def\trop{\operatorname{trop}}
\def\Trop{\operatorname{Trop}}

\def\os{\overline{S}}

\def\tN{\tilde{N}}
\def\tS{\tilde{S}}
\def\ts{\tS}
\def\ts7{\tilde{S}_7}

\def\cN{\Cal N}

\def\bP{\Bbb P}

\def\Hom{\operatorname{Hom}}

\def\tY{\tilde{Y}}

\def\cO{\Cal O}

\def\tm{\tilde{m}}

\def\codim{\operatorname{codim}}

\def\bZ{\Bbb Z}

\def\bQ{\Bbb Q}
\def\bG{\Bbb G}
\def\bR{\Bbb R}
\def\bA{\Bbb A}

\def\oX{\overline{X}}

\def\cF{\Cal F}
\def\tS{\tilde S}

\def\oS{\overline{S}}

\def\tS{\tilde{S}}

\def\oN{\overline{N}}

\def\cY{\Cal Y}
\def\cX{\Cal X}

\def\cM{\Cal M}

\def\Spec{\operatorname{Spec}}

\def\tM{\tilde{M}}

\def\ord{\operatorname{ord}}

\def\Cal{\mathcal}

\def\oX{\overline X}

\def\Exc{\operatorname{Exc}}
\def\val{\operatorname{val}}
\def\oN{\overline{N}}
\def\os7p{\oS_7'}
\def\os7{\oS_7}
\def\on6{\oN_6}
\def\n6{\oN_6}

\def\cL{\Cal L}

\newtheoremstyle{mystyle}{}{}{\itshape}{}{\scshape}{.}{ }{}
\theoremstyle{mystyle}

\swapnumbers
\newtheorem{Theorem}{Theorem}[section]

\newtheorem{Proposition}[Theorem]{Proposition}
\newtheorem{Lemma}[Theorem]{Lemma}

\newtheorem{Corollary}[Theorem]{Corollary}

\newtheorem{Question}[Theorem]{Question}

\newtheoremstyle{myreview}{}{}{}{}{\scshape}{.}{ }{}
\theoremstyle{myreview}

\newtheorem{Definition}[Theorem]{Definition}
\newtheorem{Example}[Theorem]{Example}

\newtheorem{Review}[Theorem]{}

\newcounter{et}[Theorem]
\def\cooltag{\tag{\arabic{section}.\arabic{Theorem}.\arabic{et}}\addtocounter{et}{1}}

\begin{document}
\title{On a Question of B.~Teissier}
\author{Jenia Tevelev}

\maketitle

\begin{abstract}
We answer positively a question of B.~Teissier on existence of resolution
of singularities inside an equivariant map of toric varieties.
\end{abstract}

\section{Introduction}

It is sometimes convenient to study an algebraic variety if it is embedded in a toric variety. 
B. Teissier asks in~\cite{Tei} if it is possible to perform 
resolution of singularities of an arbitrary algebraic variety inside an equivariant map of toric varieties.
The following theorem provides an affirmative answer, in fact we show that any embedded resolution
of singularities is induced by an equivariant map of toric varieties.

\begin{Theorem}\label{mainTh} Consider an embedded resolution of singularities of $X$, or more generally
any commutative diagram of irreducible projective algebraic varieties
\begin{equation}\label{zfbgsz}
\begin{matrix}
Y &\hookrightarrow &W\\
\downarrow&&{}\ \downarrow^{\pi} \\
X &\hookrightarrow &S\\
\end{matrix}
\cooltag\end{equation}
where 
\begin{itemize}
\item $W$ and $S$ are smooth;
\item $\pi$ is birational and $D:=\Exc(\pi)$ is a divisor with simple normal crossings;
\item $Y$ is smooth and intersects $D$ transversally.
\end{itemize}
Then we can extend this diagram to a commutative diagram
$$
\begin{matrix}
Y &\hookrightarrow &W &\hookrightarrow &Z\\
\downarrow&&{}\ \downarrow^{\pi}&&\downarrow \\
X &\hookrightarrow &S &\hookrightarrow &\bP^N\\
\end{matrix}
$$
where 
\begin{itemize}
\item $Z$ is smooth toric variety of an algebraic torus $\bG_m^N=\bP^N\setminus\bigcup_i H_i$ for some choice of hyperplanes $H_0,\ldots,H_N\subset\bP^N$;
\item $Z\to\bP^N$ is a toric morphism;
\item $Y$ and $W$ intersect the toric boundary of $Z$ transversally.
\end{itemize}
Moreover, we can assume that the embedding $S\hookrightarrow\bP^N$ is given by a complete linear system
associated with a sufficiently high multiple of any ample divisor on~$S$. 
\end{Theorem}

The proof is not original: it is a souped-up version of the proof by Luxton and Qu
of \cite[Theorem 1.4.]{LQ} conjectured by the author.
It is based on a criterion of Hacking, Keel, and Tevelev from \cite[\S2]{HKT},
which shows that, given a pair $(W,G)$ 
of a smooth variety $W$ and a divisor $G\subset W$ with simple normal crossings,
which satisfy certain strong but easily verified conditions (see the next section), 
there exists an embedding of $W$ into a smooth toric variety $Z$
such that $G$ is the scheme-theoretic intersection of $W$ with the toric boundary of~$Z$.

In practice we may want to say more about $Z$. We can easily make $Z$ proper
by adding toric strata that don't intersect $W$. 
Applying a theorem of De Concini and Procesi \cite[Theorem 2.4]{CP}, it is easy to prove that

\begin{Corollary}\label{corollar}
Let $X$ be an irreducible subvariety of a smooth projective variety $S$
(over an algebraically closed 
field of characteristic $0$).
Then there exist
\begin{itemize}
\item a projective embedding $S\subset\bP^N$ (given by a sufficiently high multiple
of any ample divisor on $S$);
\item coordinate hyperplanes $H_0,\ldots,H_N\subset\bP^N$ such that $X\not\subset H_0\cup\ldots\cup H_N$;
\item a smooth projective toric variety $Z$ of an algebraic torus $\bG_m^N=\bP^N\setminus\bigcup_i H_i$
\end{itemize}
such that
\begin{itemize}
\item a toric morphism $Z\to \bP^N$ is a composition of blow-ups in smooth equivariant centers of codimension~$2$;
\item a proper transform $W$ of $S$ in $Z$ is smooth, intersects toric boundary of $Z$ transversally, and in particular
$D=\Exc(W\to S)$ is a divisor with simple normal crossings.
\item a proper transform $Y$ of $X$ in $Z$ is smooth and intersects the toric boundary of $Z$ (and in particular $D$) transversally.
\end{itemize}
\end{Corollary}

The trade-off in this corollary is that 
the resolution of singularities $Y\to X$ could fail to be an isomorphism over a smooth locus of $X$.

I am grateful to B.~Teissier for explaining his conjecture and partial results on resolutions 
in ambient toric varieties and to the organizers of the ``Toric Geometry'' workshop at 
Mathematisches Forschungsinstitut Oberwolfach where these conversations took place.
The research was supported by NSF grant DMS-1001344.

\section{Proof of Theorem~\ref{mainTh} and Corollary~\ref{corollar}}

Essentially we would like to prove that $(W,D)$
embeds in a toric variety in such a way that $D$ is a scheme-theoretic intersection
with the toric boundary. This is not true in general, but the
 main idea is that this is going to work after we add a lot of random divisors to $D$. 

 Let $D_1,\ldots,D_m$ be irreducible components of $D$ and let
 $D_0:=\emptyset$.
 Choose an invertible sheaf $\cL$ on $W$ such that $\cL(D_i)$ is very ample for any~$i\ge0$.
Then
$$\cL\simeq\pi^*\cM\left(-\sum_{i=1}^m a_iD_i\right)$$
for some line bundle $\cM$ on $S$. By tensoring $\cM$ with an appropriate very ample line bundle on $S$,
we can arrange that $\cM$ is very ample and moreover is isomorphic to a tensor power of any given
ample line bundle on $S$. By Kodaira lemma~\cite[2.19]{FA},  $a_i>0$ for any $i$.
Let 
\begin{equation}\label{sdgbdbdbdgh}
\alpha=2\dim(W)-1+\max_{0\le i\le m} h^0(W, \cL(D_i))\quad\hbox{\rm and}\quad
r=\alpha(m+1)
\cooltag\end{equation}
Let 
$$F_0,\ldots,F_{r-1}\subset W$$
be divisors obtained by taking $\alpha$ general divisors from each linear system $|\cL(D_i)|$ for  $i=0,\ldots,m$.
Let 
$$s_0,\ldots,s_{r-1}\in H^0(S,\cM)$$ be equations
of divisors $\pi(F_0),\ldots,\pi(F_{r-1})\subset S$. Since $F_i$'s are general,
$\pi(F_i)\not\subset\pi(F_j)$ for $i\ne j$, and therefore sections
$$z_i:=s_0\ldots\hat s_i\ldots s_{r-1}\in H^0(S,\cM^{\otimes (r-1)}), \quad i=0,\ldots,r-1,$$
are linearly independent. Add general sections $z_{r},\ldots,z_N$ so that $z_0,\ldots,z_N$ is a basis of   
$H^0(S,\cM^{\otimes (r-1)})$.
Let $E_{r},\ldots,E_N\subset W$ be pull-backs of the hypersurfaces $(z_{r}=0),\ldots,(z_N=0)\subset S$.
We let 
$$G=D_1+\ldots+D_m+F_0+\ldots+F_{r-1}+E_{r}+\ldots+E_N\subset W.$$
By Bertini theorem, all components of $G$ are irreducible, smooth, and have simple normal crossings.
Also, $Y$ and $W$ intersect $G$ transversally.

Consider the embedding 
$S\hookrightarrow\bP^N$ given by a complete linear system of $\cM^{\otimes (r-1)}$
and homogeneous coordinates $z_0,\ldots,z_N$.
Let $\bG_m^N$ be the corresponding torus.
Since $D=\Exc(\pi)$ and $F_i$ is ample, we have $\pi(F_i)\supset \pi(D)$ for any $i=0,\ldots,r-1$.
Therefore, the map $\pi$ induces an isomorphism
$$W\setminus G\simeq S\cap\bG_m^N.$$
Let $I$ be the indexing set for irreducible components of $G$, so we have
$$G=\sum_{i\in I}G_i=D_1+\ldots+D_m+F_0+\ldots+F_{r-1}+E_{r}+\ldots+E_N.$$
Let 
$$M\subset\cO^*(W\setminus G)$$ 
be a sublattice
generated by $z_i/z_j$ for $i,j=0,\ldots,N$.

\begin{Lemma}\label{avoid}
Let $J\subset I$, $|J|\le 2\dim(W)-1$, and let $i\in I\setminus J$.
 Then there exists a subset $T\subset I\setminus J$ such that $i\in T$ and
 \begin{itemize}
 \item $U=W\setminus \bigcup\limits _{t\in T} G_t$ is affine and $\cO(U)$ is generated by $M\cap \cO(U)$.
 \item There exists $m\in M$ such that $\val_{G_i}m=1$, $\val_{G_j}m=0$ for any $j\in J$.
 \end{itemize}
\end{Lemma}

\begin{proof} 
 We consider three cases.

{\bf Case I: $G_i$ is an F-type divisor, i.e.~$G_i\in |\cL(D_p)|$ for some $p$.}
By definition~\eqref{sdgbdbdbdgh}  of $\alpha$, we can  choose a subset $T=\{i,k_1,\ldots,k_q\}\subset I\setminus J$ such that
$G_i,  G_{k_1},\ldots,G_{k_q}$ is a basis of $|\cL(D_p)|$.
Then $U$ is affine as a closed subvariety
of an algebraic torus $\bG_m^{q}$ (the complement to the union 
of coordinate hyperplanes in $\bP^{q}=\bP H^0(W,\cL(D_p))^\vee$).
Moreover, $\cO(U)$ is generated by ratios of coordinate functions in $\bP^{q}$.
As a rational function on~$W$, such a function $f$ has a simple zero at $G_i$, a simple pole at some~$G_{k_s}$,
and is invertible along other components of $G$.
Notice that the function $f_0={s_{i-m-1}\over s_{k_s-m-1}}={z_{k_s-m-1}\over z_{i-m-1}}$
has the same property, except that apriori it may also have some zeros and poles along the exceptional divisor $D$.
But then $f/f_0$ has zeros and poles only along the exceptional divisor $D$, and so it must be a constant.
Thus, any of these rational functions can be used as $m$.

{\bf Case II: $G_i$ is a D-type divisor, i.e.~$G_i=D_p$ for some $p$.}
Choose a subset $T=\{i,j,k_1,\ldots,k_q\}\subset I\setminus J$ such that $G_j\in|\cL|$,
$G_{k_1},\ldots,G_{k_q}\in|\cL(D_p)|$ and 
$G_i+G_j,  G_{k_1},\ldots,G_{k_q}$ is a basis of $|\cL(D_p)|$.
Then $U$ is affine as a closed subvariety
of an algebraic torus $\bG_m^{q}$ (the complement to the union 
of coordinate hyperplanes in $\bP^{q}=\bP H^0(W,\cL(D_p))^\vee$).
Moreover, $\cO(U)$ is generated by ratios of coordinate functions in $\bP^{q}$.
As a rational function on $W$, such a function has a simple zero at $D_p$ and $G_j$, a simple pole at some $G_{k_s}$,
and is invertible along other components of $G$.
As in Case I, it follows that this function must be equal to ${s_{j-m-1}\over s_{k_s-m-1}}={z_{k_s-m-1}\over z_{j-m-1}}$ 
(up to a scalar multiple).
We can take any of them as $m$.

{\bf Case III: $G_i$ is an E-type divisor, i.e.~$G_i=E_p$ for some $p$.} Let $n=\dim W$.
By Riemann--Roch, we can substitute $\cM$ with its tensor power if necessary to ensure that $N-r>3n$.
It follows that we can find a subset
$T=\{i_0,\ldots,i_q,j_1,\ldots,j_{n+1}\}\subset I\setminus J$, where
$G_{i_0},\ldots,G_{i_q}$ forms a basis of $|\cL|$, $G_{j_1},\ldots,G_{j_{n+1}}$ are $E$-type divisors, and $G_{j_1}=E_p$.
Let $\bP^{q}=\bP H^0(W,\cL)^\vee$ with coordinates that correspond to $G_{i_0},\ldots,G_{i_q}$.
Since $G$ is a divisor with normal crossings, 
$G_{j_1}\cap\ldots\cap G_{j_{n+1}}=\emptyset$. It follows that $S\subset\bP^N$
misses the intersection of the corresponding  $n+1$ coordinate hyperplanes.
Projecting from this subspace gives a morphism $W\to\bP^n$, where $\bP^n$ has coordinate hyperplanes that correspond
to $G_{j_1},\ldots,G_{j_{n+1}}$. Consider a diagonal embedding
$$W\hookrightarrow\bP^q\times\bP^n.$$
Then $U$ is naturally a closed subvariety
of an algebraic torus $\bG_m^{q}\times\bG_m^n$ and $\cO(U)$ 
is generated by ratios of coordinate functions in $\bP^{q}$ and ratios of coordinate functions in $\bP^n$.
Arguing as in the previous cases, these ratios are equal to some $z_\alpha/z_\beta$, where $\alpha,\beta< r$ (in case of $\bP^q$)
and $\alpha,\beta\ge r$ (in case of $\bP^n$). One of the latter ones can be used as~$m$.
\end{proof}

\begin{Lemma}\label{sortofmainlemma}
$W$ can be embedded in a smooth toric variety $Z$ of $\bG_m^N$ in such a way that
$G$ is a scheme-theoretic intersection with the toric boundary. Moreover, intersecting with $W$
induces a bijection between toric divisors of $Z$ and irreducible components of $G$.
A collection of toric divisors has a non-empty intersection if and only if the corresponding
components of $G$ have a non-empty intersection.
\end{Lemma}

\begin{proof}
For any subset $S\subset I$, let 
$$W_S:=\bigcap_{i\in S} G_i\quad\hbox{\rm and}\quad U_S:=\bigcap_{i\not\in S}(W\setminus G_i).$$
In particular, we have
$$W_\emptyset:=W\quad\hbox{\rm and}\quad U_\emptyset=W\setminus G.$$
If $W_S$ is non-empty, we call it a  {\em stratum} of $W$ (which could be reducible).
For any $S\subset I$,  let 
$M_S=M\cap\cO(U_S)$.
By \cite[\S2]{HKT}, Lemma~\ref{sortofmainlemma} will follow if we can 
check the following three conditions:
  
 \begin{enumerate}
 \item For any stratum $W_S$, 
 $U_S$ is affine and $M_S$ generates $\cO(U_S)$.
  \item For any stratum $W_S$, and any $i\in S$, there exists $m\in M$
 with $\val_{G_i}m=1$ and $\val_{G_j}m=0$ for any $j\in S\setminus\{i\}$.
 \item The collection of cones in $M^\vee\otimes_\bZ\bQ$ convexly dual to cones spanned by semi-groups
 $M_S$ forms a smooth fan
 (as $W_S$ runs over strata).
 \end{enumerate}
 
 To check {\rm(1)}, we use $J=S$ in Lemma~\ref{avoid}, which then shows that
 $S$ can be written as an intersection of subsets $K_\alpha$ such that 
 $U_{K_\alpha}$ 
 is affine and $\cO(U_{K_\alpha})$ is generated by $M_{K_\alpha}$.
 By separatedness, $U_S=\bigcap_\alpha U_{K_\alpha}$
 is affine as well and  $\cO(U_S)$ is generated by restrictions
of $\cO(U_{K_\alpha})$ for all~$\alpha$,
hence by the union of $M_{K_\alpha}$ for all~$\alpha$, and
hence by $M_S$. 
To check {\rm(2)}, take $J=S\setminus\{i\}$  in Lemma~\ref{avoid}.
Finally, to check {\rm(3)}, it suffices to show 
that for any two strata $W_{S_1}$, $W_{S_2}$, there exists a unit $m\in M$ such that $\val_{G_i}m=1$ for any $i\in S_1\setminus S_2$ and $\val_{G_j}m=0$ for each $j\in S_2$.  
For each $i\in S_1\setminus S_2$, we apply Lemma~\ref{avoid} to $J=S_1\cup S_2\setminus\{i\}$, which gives a unit $m_i$
such that $\val_{G_i}m_i=1$ and $\val_{G_j}m_i=0$ for any $j\in J$. But then $m=\prod\limits_{i\in S_1\setminus S_2}m_i$ 
satisfies (3).
\end{proof}

It remains to show that a rational equivariant map $\phi:\,Z\dashrightarrow\bP^N$
is in fact a morphism. Let $N=M^\vee\otimes_\bQ\bR$. Let $C\subset N$ be a cone in the fan of $Z$
with rays that correspond to toric divisors which cut out divisors $G_i$, $i\in J$ on $W$ for some subset $J\subset I$.
Then $|J|\le\dim W$, and arguing as in Case III of Lemma~\ref{avoid}, we can find a prime divisor $G_k$ of type $E$ 
such that $k\not\in J$. This divisor corresponds to one of the coordinate hyperplanes $H\subset \bP^N$. 
Let $D\subset N$ be a cone in the fan of $\bP^N$ with rays that correspond to coordinate hyperplanes other than $H$.
It suffices to show that $C\subset D$. Dually, it suffices to show inclusion of semigroups $D^\vee\subset C^\vee$.
Let $m\in D^\vee$. Then 
$$m\in M\subset k(S)=k(W).$$
The principal divisor $(m)$ on $S$ has only one component with negative multiplicity, namely $H\cap S$.
Since $H$ is a general hyperplane, $\pi(G_j)\not\subset H$ for any $j\in J$.
It follows that $\ord_{G_j}m\ge0$ for any $j\in J$ and therefore $m\in C^\vee$.

This finishes the proof of Theorem~\ref{mainTh}. Now we prove Corollary~\ref{corollar}.
We first apply Hironaka's resolution of singularities \cite{Hi} to construct an embedded resolution of singularities \eqref{zfbgsz}.
Then we apply Theorem~\ref{mainTh}, and compactify $Z$ to a smooth proper toric variety, which we also call $Z$.
Now we apply \cite[Theorem 2.4]{CP}, which gives a proper toric morphism $\tilde Z\to Z$ such that 
the morphism $\tilde Z\to\bP^N$ is a composition of blow-ups
with smooth equivariant centers of codimension~$2$.

It remains to show that the proper transform $\tilde Y$ of $X$ in $\tilde Z$ is smooth 
and intersects the toric boundary transversally (the proof of the corresponding facts for $S$ is similar). 
This follows from \cite[Th~1.4 and Prop.~2.5]{Te}.
More precisely, these results apply as follows:
since $Z$ is smooth and $Y$ is smooth
and intersects the toric boundary transversally, the multiplication map 
$$\Psi:\,Y\times \bG_m^N\to Z$$
is smooth. 
The multiplication map 
$\tilde\Psi:\,\tilde Y\times \bG_m^N\to \tilde Z$
is then a pull-back of $\Psi$, and therefore it is also smooth.
Since $\tilde Z$ is smooth, this finally implies that 
$\tilde Y$ is smooth and intersects the toric boundary transversally.

\end{document}